\documentclass{amsart}

\usepackage{amsmath,amsxtra,amsthm,amssymb,amsfonts} 

\theoremstyle{plain}

\newtheorem{thm}{Theorem}[section]

\newtheorem{defi}[thm]{Definition}
\newtheorem{prop}[thm]{Proposition}
\newtheorem{lem}[thm]{Lemma}

\newtheorem{notation}[thm]{Notation}

\theoremstyle{definition}

\newtheorem{rem}[thm]{Remark}
\newtheorem{exa}[thm]{Example}

\def\nat{\mathbb{N}}
\def\rls{\mathbb{R}}

\def\exrls{(-\infty,\infty]}

\def\wto{\stackrel{w}{\to}}
\def\ol{\overline}

\def\argmin{\operatornamewithlimits{\arg\min}}

\def\dom{\operatorname{dom}}

\def\cldom{\operatorname{\ol{dom}}}

\def\col{\colon}
\def\l{\left}
\def\r{\right}

\def\hs{\mathcal{H}}

\def\as{\mathrel{\mathop:}=}      

\def\fun{\col\hs\to\exrls}
\def\map{\col\hs\to\hs}
\def\geo{\col[0,1]\to\hs}    

\begin{document}
\title[The {L}ie-{T}rotter-{K}ato formula]{A new proof of the {L}ie-{T}rotter-{K}ato formula in~{H}adamard spaces}
\author[M. Ba\v{c}\'ak]{Miroslav Ba\v{c}\'ak}
\date{\today}
\subjclass[2010]{47H20; 58D07}
\thanks{The research leading to these results has received funding from the
 European Research Council under the European Union's Seventh Framework
 Programme (FP7/2007-2013) / ERC grant agreement no 267087.}
\keywords{Gradient flow semigroup, {H}adamard space, {L}ie-{T}rotter-{K}ato formula, resolvent, weak convergence.}
\address{Miroslav Ba\v{c}\'ak, Max Planck Institute, Inselstr. 22, 04 103 Leipzig, Germany}
\email{bacak@mis.mpg.de}

\begin{abstract}
The Lie-Trotter-Kato product formula has been recently extended into Hadamard spaces by [Stojkovic, Adv. Calc. Var., 2012]. The aim of our short note is to give a simpler proof relying upon weak convergence instead of an ultrapower technique.
\end{abstract}

\maketitle



\section{Introduction}
Let $f\fun$ be a convex lower semicontinuous function (lsc) defined on an Hadamard space $(\hs,d).$ For instance $\hs$ can be a Hilbert space and $d$ its natural metric induced by the inner product. For the notation and terminology not explained here, the reader is referred to Section~\ref{sec:pre}. Given $\lambda>0,$ define the \emph{resolvent} of $f$ as
\begin{equation} \label{eq:defres}
J_\lambda(x) \as\argmin_{y\in \hs} \l[f(y)+\frac1{2\lambda}d(x,y)^2\r],\qquad x\in\hs.
\end{equation}
The mapping $J_\lambda\col\hs\to\cldom f$ is well-defined for each $\lambda\in(0,\infty);$ see \cite[Lemma~2]{jost95} and \cite[Theorem~1.8]{mayer}. We also put $J_0(x)\as x$ for each $x\in \hs.$ The \emph{gradient flow semigroup} of $f$ is given as
\begin{equation} \label{eq:defsem}
S_t (x)\as \lim_{n\to\infty} \l(J_{\frac{t}n}\r)^{(n)} (x),\qquad x\in \cldom f,
\end{equation}
for every $t\in[0,\infty).$ The limit in~\eqref{eq:defsem} is uniform with respect to $t$ on bounded subintervals of $[0,\infty)$ and $\l(S_t\r)_{t\geq0}$ is a strongly continuous semigroup of nonexpansive mappings on~$\hs;$ see \cite[Theorem~1.3.13]{jost-ch} and \cite[Theorem~1.13]{mayer}. Note that formula~\eqref{eq:defsem} was in a similar context used already in~\cite[Theorem 8.2]{reichshafrir}.

\begin{rem}
If $\hs$ is a Hilbert space, $u_0\in\cldom f$ and we put $u(t)\as S_t\l(u_0\r),$ for $t\in[0,\infty),$ we obtain a ``classical'' solution to the parabolic problem
\begin{equation*} \label{eq:gradflowproblem}
\dot{u}(t)\in - \partial f\l(u(t)\r),\qquad t\in(0,\infty),
\end{equation*}
with the initial condition $u(0)=u_0.$
\end{rem}

In the present paper, we consider a function $f\fun$ of the form 
\begin{equation} \label{eq:f}
f\as\sum_{j=1}^k f_j,
\end{equation}
where $f_j\fun$ are convex lsc functions, $j=1,\dots,k$ and $k\in\nat.$ This covers a surprisingly large spectrum of problems and has become a classical framework in various applications; see for instance~\cite[Proposition 27.8]{baucom} for the so-called parallel splitting algorithm. In the Hadamard space setting, functions of the form~\eqref{eq:f} naturally emerged in connection with the following example.
\begin{exa} \label{exa:sumofdtothep}
Given a finite number of points $a_1,\dots,a_k\in \hs$ and positive weights $w_1,\dots,w_k$ with $\sum_{j=1}^k w_j=1,$ we define the function
\begin{equation*} f(x)\as \sum_{j=1}^k w_j d\l(x,a_j\r)^p,\qquad x\in\hs,\end{equation*}
where $p\in[1,\infty).$ Then~$f$ is convex continuous and we are especially interested in two important cases:
\begin{enumerate}
 \item If $p=1,$ then $f$ becomes the objective function in the \emph{Fermat-Weber problem} for optimal facility location. If, moreover, all the weights $w_n=\frac1k,$ a minimizer of~$f$ is called a \emph{median} of the points $a_1,\dots,a_k.$
 \item If $p=2,$ then a minimizer of~$f$ is the \emph{barycenter} of the probability measure
\begin{equation*} \mu=\sum_{j=1}^k w_j \delta_{a_j},\end{equation*}
where $\delta_{a_j}$ stands for the Dirac measure at the point $a_j.$ For further details on barycenters, the reader is referred to~\cite[Chapter 3]{jost2} and~\cite{sturm}. If, moreover, all the weights $w_j=\frac1k,$ the (unique) minimizer of~$f$ is called the \emph{Fr\'echet mean} of the points $a_1,\dots,a_k.$
\end{enumerate}
\end{exa}

Both medians and means of points in an Hadamard space are currently a subject of intensive research for their applications in computational biology; see \cite{median-mean,benner,miller} and the references therein. Another area where Fr\'echet means play an important role is the study of so-called consensus algorithms~\cite{grohs}.

Having demonstrated the importance of functions of the form~\eqref{eq:f}, we will now turn to a nonlinear version of the Lie-Trotter-Kato product formula, which states that the gradient flow semigroup of~$f$ can be approximated by the resolvents or semigroups of the individual functions~$f_j,$ with $j=1,\dots,k.$ Such approximation results for semigroups of (much more general) operators defined on a Banach space have become a classical part of functional analysis and go back to the seminal works of Brezis and Pazy~\cite{brezis-pazy70,brezis-pazy72}, Chernoff~\cite{chernoff}, Kato~\cite{kato}, Miyadera and {\^O}haru~\cite{miyadera}, Reich~\cite{reich82,reich83}, Trotter~\cite{trotter,trotter2} and others. For the details, see also the monographs~\cite{brezis-book,chernoff-book,en,pazy}. We shall need some notation.
\begin{notation} \label{not:s}
Let $f\fun$ be a function of the form~\eqref{eq:f} and we of course assume that it is not identically equal to $\infty.$ Its resolvent $J_\lambda$ and semigroup $S_t$ are given by~\eqref{eq:defres} and~\eqref{eq:defsem}, respectively. The resolvent of the function $f_j$ is denoted $J_\lambda^j,$ for each $j=1,\dots,k,$ and likewise the semigroup of~$f_j$ is denoted $S_t^j,$ for each $j=1,\dots,k.$ The symbol $P_j$ will denote the metric projection onto $\cldom f_j,$ where $j=1,\dots,k.$
\end{notation}
If $F\map$ is a mapping, we denote its $n^{\mathrm{th}}$ power, with $n\in\nat,$ by
\begin{equation*}F^{(n)}x\as\l(F\circ\dots\circ F\r) x,\qquad x\in\hs,\end{equation*}
where $F$ appears $n$-times on the right-hand side. As a convention, we set $F^{(0)}x\as x$ for every $x\in\hs.$ Having this notation at hand, we are able to state a nonlinear version of the Lie-Trotter-Kato product formula due to Stojkovic~\cite[Theorems 4.4 and 4.8]{stoj}. 
\begin{thm} \label{thm:stoj}
Let $(\hs,d)$ be an Hadamard space and $f\fun$ be of the form~\eqref{eq:f}. We use Notation~\ref{not:s}. Then, for every $t\in[0,\infty)$ and $x\in\cldom f,$ we have
\begin{subequations}
\begin{align}
S_t(x) & =  \lim_{n\to\infty} \l(J_{\frac{t}{n}}^k\circ\dots\circ J_{\frac{t}{n}}^1\r)^{(n)}(x) , \label{i1:stoj:i} \\
\intertext{and, if moreover there exists $1\leq l\leq k$ such that $\dom f_l$ is locally compact, then}
S_t(x) & =  \lim_{n\to\infty} \l(S_{\frac{t}{n}}^k\circ P_k \dots\circ S_{\frac{t}{n}}^1\circ P_1\r)^{(n)}(x) . \label{i1:stoj:ii}
\end{align}
\end{subequations}
\end{thm}
The goal of the present paper is to reprove this theorem. In the proof of~\eqref{i1:stoj:i}, we employ weak convergence instead of an ultralimit technique used in~\cite[Theorems 4.4]{stoj}, which in our opinion simplifies the original proof. Part~\eqref{i1:stoj:ii} was in~\cite[Theorem 4.8]{stoj} proved with the help of~\cite[Lemma 4.3]{stoj}, but since this lemma holds only for $z\in\cldom f$ and we do not know apriori that $x_\lambda\in\cldom f,$ a more careful argument is needed to fill this gap. The question whether~\eqref{i1:stoj:ii} holds even without the additional local compactness assumption is left open. We also note that the proof of Theorem~\ref{thm:stoj} relies upon \cite[Theorem 3.13]{stoj} and one should therefore obtain \emph{uniform} convergence of both~\eqref{i1:stoj:i} and~\eqref{i1:stoj:ii} with respect to $t$ on each bounded subinterval of $[0,\infty).$ We were however unable to follow the argument leading to the \emph{uniform} convergence in \cite[Theorem 3.13]{stoj} (more precisely in \cite[Theorem 3.12]{stoj}).

To demonstrate that the gradient flow theory in Hadamard spaces applies in various situations, we now present a number of natural examples of convex lsc functions on an~Hadamard space $(\hs,d).$
\begin{exa}[Indicator functions] \label{exa:indicator}
Let $K\subset \hs$ be a convex set. Define the~\emph{indicator function} of $K$ by
\begin{equation*} \iota_K(x)\as\l\{
\begin{array}{ll} 0, & \text{if } x\in K,  \\ \infty,  & \text{if } x\notin K. \end{array} \r. \end{equation*}
Then $\iota_K$ is a convex function and it is lsc if and only if $K$ is closed.
\end{exa}

\begin{exa}[Distance functions] \label{exa:dist}
Given a point $x_0\in\hs,$ the function
\begin{equation}\label{eq:dist} x\mapsto d\l(x,x_0\r),  \qquad x\in \hs,\end{equation}
is convex and continuous. The function $d\l(\cdot,x_0\r)^p$ for $p>1$ is strictly convex. More generally, the \emph{distance function} to a closed convex subset $C\subset \hs,$ defined by
\begin{equation*}d_C(x)\as \inf_{c\in C} d(x,c),  \qquad x\in \hs,\end{equation*}
is convex and $1$-Lipschitz~\cite[p.178]{bh}.
\end{exa}
\begin{exa}[Displacement functions]
Let $T\map$ be an isometry. The \emph{displacement function} of $T$ is the function $\delta_T\col\hs\to[0,\infty)$ defined by
\begin{equation*}\delta_T(x)\as d(x,Tx),\qquad x\in  \hs.\end{equation*}
It is convex and Lipschitz~\cite[p.229]{bh}.
\end{exa}
\begin{exa}[Busemann functions] \label{exa:busemann}
Let $c\col[0,\infty)\to\hs$ be a geodesic ray. The function $b_c\col\hs\to\rls$ defined by
\begin{equation*} b_c(x)\as \lim_{t\to\infty} \l[d\l(x,c(t) \r) -t \r],  \qquad x\in\hs,\end{equation*}
is called the \emph{Busemann function} associated to the ray $c.$ Busemann functions are convex and $1$-Lipschitz. Concrete examples of Busemann functions are given in \cite[p.~273]{bh}. Another explicit example of a Busemann function in the Hadamard space of positive definite $n\times n$ matrices with real entries can be found in~\cite[Proposition~10.69]{bh}. The sublevel sets of Busemann functions are called \emph{horoballs} and carry a lot of information about the geometry of the space in question; see \cite{bh} and the references therein.
\end{exa}
\begin{exa}[Energy functional]
The energy functional is another important instance of a convex function on an Hadamard space \cite{jost94,jost95,jost97,ks}. Indeed, the energy functional is convex and lsc on a suitable Hadamard space of $\mathcal{L}^2$-mappings. Minimizers of the energy functional are called \emph{harmonic maps} and are important in both geometry and analysis. For a probabilistic approach to harmonic maps in Hadamard spaces, see~\cite{sturm-markov1,sturm-markov2,sturm-semigr}.
\end{exa}

We shall finish this Introduction by recalling a brief development of the theory of gradient flows in Hadamard spaces, which has recently attracted considerable interest. It started independently by the work of Jost~\cite{jost-ch} and Mayer~\cite{mayer}, when the existence of the gradient flow semigroup was established. The study of the relationship with the Mosco and $\Gamma$- convergences, initiated already in~\cite{jost-ch}, was treated in greater detail in~\cite{semigroups,japan}. In~\cite{ppa}, the author describes large time behavior of the gradient flow as well as its discrete version called the proximal point algorithm. As already mentioned above, the Lie-Trotter-Kato product formula was proved in~\cite{stoj}. There have also been many related results in some special instances of Hadamard spaces, namely, in manifolds of nonpositive sectional curvature (\cite{sevilla,sevilla2,papa}) and the Hilbert ball (\cite{kop,kr2010,reichshoikhet} and the references therein). On the other hand this theory can be 
partially extended into more general metric spaces and plays an important role in optimal transport theory, PDEs and probability~\cite{ambrosio}. For another viewpoint, see~\cite{genaro}.

\vspace{12pt}

\paragraph{Acknowledgments:} I would like to express my gratitude to the referee for his comments and suggestions.



\section{Preliminaries} \label{sec:pre}
We first recall basic notation and facts concerning Hadamard spaces. For further details on the subject, the reader is referred to \cite{bh}. 

\subsection{Hadamard spaces}
If a geodesic metric space $(X,d)$ satisfies the inequality 
\begin{equation} \label{eq:cat}
 d\l(x,\gamma(t)\r)^2\leq (1-t)d\l(x,\gamma(0)\r)^2+td\l(x,\gamma(1)\r)^2-t(1-t)d\l(\gamma(0),\gamma(1)\r)^2,
\end{equation}
for any $x\in X,$ any geodesic $\gamma\col[0,1]\to X$ and any $t\in[0,1],$ we say it has nonpositive curvature (in the sense of Alexandrov), or that it is a CAT(0) space. A~complete CAT(0) space is called an \emph{Hadamard space.}

The class of Hadamard spaces includes Hilbert spaces, $\rls$-trees, Euclidean Bruhat-Tits buildings, classical hyperbolic spaces, complete simply connected Riemannian manifolds of nonpositive sectional curvature, the Hilbert ball, CAT(0) complexes and other important spaces included in none of the above classes~\cite{bh}.

Let $(\hs,d)$ be an Hadamard space. Having two points $x,y\in \hs,$ we denote the geodesic segment from $x$ to $y$ by $[x,y].$ We usually do not distinguish between a~geodesic and its geodesic segment, as no confusion can arise. For a point $z\in[x,y],$ we write $z=(1-t)x+ty,$ where $t=\frac{d(x,z)}{d(x,y)}.$ 

For a function $f\fun,$ we denote $\dom f\as\l\{x\in \hs\col f(x)<\infty\r\}.$ If $\dom f\neq\emptyset,$ we say that $f$ is \emph{proper.} To avoid trivial situations we often assume this property without mentioning it explicitly. As usual, the symbol $\cldom f$ stands for the closure of $\dom f.$ A point $x\in\hs$ is called a \emph{minimizer} of $f$ if $f(x)=\inf_\hs f.$

\subsection{Convex sets and functions on Hadamard spaces} \label{subsec:prelim} Recall that a set $C\subset \hs$ is \emph{convex} if $x,y\in C$ implies $[x,y]\subset C.$ A~function $f\fun$ is \emph{convex} provided $f\circ\gamma\col[0,1]\to\exrls$ is convex for each geodesic $\gamma\geo.$ Note that the distance function $d_C$ is convex and continuous; see Example~\ref{exa:dist}.
\begin{prop}
Let $(\hs,d)$ be an Hadamard space and $C\subset \hs$ be a closed convex set. Then:
\begin{enumerate}
\item For every $x\in \hs,$ there exists a unique point $P_C(x)\in C$ such that
\begin{equation*}d\l(x,P_C(x)\r)=d_C(x).\end{equation*}
\item If $x\in\hs$ and $y\in C,$ then
\begin{equation} \label{eq:pyth}
d(x,y)^2\geq d\l(x,P_C(x)\r)^2+d\l(y,P_C(x)\r)^2
\end{equation}
\item The mapping $P_C\col\hs\to C$ is nonexpansive and is called the \emph{metric projection} onto~$C.$
\end{enumerate}
\label{prop:proj}
\end{prop}
\begin{proof} See \cite[Proposition~2.4, p.176]{bh}. \end{proof}

The following result comes from~\cite[Lemma 2.2]{gelander}. Note that the Hilbert ball version appeared in~\cite[Theorem 18.1]{goebelreich}. We include its short proof for the readers' convenience. 
\begin{prop}\label{prop:intersectprop}
 Let $(\hs,d)$ be an Hadamard space. If $\l(C_\alpha\r)_{\alpha\in I}$ is a nonincreasing family of bounded closed convex sets in $\hs.$ Then $\bigcap_{\alpha\in I} C_\alpha\neq\emptyset.$
\end{prop}
\begin{proof}
Choose $x\in\hs$ and denote its projection onto $C_\alpha$ by $x_\alpha\as P_{C_\alpha}(x).$ Then $\l(d\l(x,x_\alpha\r)\r)_\alpha$ is a nondecreasing net of nonnegative numbers and hence has a limit~$l.$ If $l=0,$ then $x\in\cap_\alpha C_\alpha.$ If $l>0,$ then we claim that $\l(x_\alpha\r)$ is Cauchy. Indeed, denote $x_{\alpha\beta}\as\frac12x_\alpha +\frac12x_\beta$ and apply~\eqref{eq:cat} with $t\as\frac12$ to obtain
\begin{equation*} d(x,x_{\alpha\beta})^2\leq\frac12 d(x,x_\alpha)^2+\frac12 d(x,x_\beta)^2-\frac14d(x_\alpha,x_\beta)^2,\end{equation*}
which implies that $(x_\alpha)$ is Cauchy. The limit point clearly lies in $\bigcap_\alpha C_\alpha.$
\end{proof}
As a consequence of Proposition~\ref{prop:intersectprop} we obtain that convex lsc functions are locally bounded.
\begin{lem}\label{lem:convexbnd}
 Let $(\hs,d)$ be an Hadamard space and $f\fun$ be a convex lsc function. Then $f$ is bounded from below on bounded sets.
\end{lem}
\begin{proof}
Let $C\subset \hs$ be bounded and without loss of generality assume that $C$ is closed convex. If $\inf_C f=-\infty,$ then the sets $S_N\as\l\{ x\in C\col f(x)\leq -N\r\}$ for $N\in\nat$ are all nonempty. Since all $S_N$ are closed convex and bounded, Proposition~\ref{prop:intersectprop} yields a point $z\in \bigcap_{N\in\nat}S_N.$ Clearly $f(z)=-\infty,$ which is not possible.
\end{proof}
The following lemma says that convex functions do not decay too fast. The proof mimics~\cite[Lemma 1.5]{attouch79}.
\begin{lem} \label{lem:growth}
 Let $(\hs,d)$ be an Hadamard and $f\fun$ be a convex lsc function. For each $x_0\in\hs$ there exist constants $\alpha,\beta\in\rls$ such that
\begin{equation*} f(x)\geq \alpha+\beta d\l(x,x_0\r),\end{equation*}
for every $x\in\hs.$
\end{lem}
\begin{proof}
Assume that this is not the case, that is, for each $n\in\nat,$ there exist $x_n\in\hs$ such that
\begin{equation*}f(x_n)<-n\l[d\l(x_n,x_0\r)+1\r].\end{equation*}
Then we have
\begin{equation*}\liminf_{n\to\infty} f(x_n)\leq-\limsup_{n\to\infty} n\l[d\l(x_n,x_0\r)+1\r]\leq-\infty,\end{equation*}
which via Lemma~\ref{lem:convexbnd} implies that $(x_n)$ is unbounded. Choose $y\in\dom f$ and put
\begin{equation*}z_n\as(1-t_n)y+t_n x_n,\quad\text{with } t_n\as\frac1{\sqrt{n}d(y,x_n)}.\end{equation*}
Then $z_n\to y.$ By convexity,
\begin{align*}
 f(z_n) & \leq (1-t_n)f(y)+t_n f(x_n) \\
        & \leq (1-t_n)f(y)-t_n n\l[d(x_n,x_0)+1\r] \\
        & \leq (1-t_n)f(y)-\sqrt{n}\frac{d(x_n,x_0)+1}{d(x_n,y)}.
\end{align*}
Thus, by lower semicontinuity, we get
\begin{equation*}f(y)\leq\liminf_{n\to\infty} f(z_n)\leq -\infty,\end{equation*}
which is not possible.
\end{proof}

Let $(\hs,d)$ be an Hadamard space, $f\fun$ be convex lsc and $x\in\hs.$ Then
\begin{equation} \label{eq:resestim}
\frac1{2\lambda}d\l(J_\lambda x,v\r)^2-\frac1{2\lambda}d\l(x,v\r)^2+\frac1{2\lambda}d\l(x,J_\lambda x\r)^2+f\l(J_\lambda x\r)\leq f(v),
\end{equation}
for each $v\in\dom f.$ See~\cite[Theorem 4.1.2]{ambrosio}. Furthermore, if $x\in\cldom f$ and we set $x(t)\as S_t x,$ for $t\in[0,\infty),$ then
\begin{equation} \label{eq:evi}
\frac1{2t}d\l(x(t),v\r)^2-\frac1{2t}d\l(x(0),v\r)^2+f\l(x(t)\r)\leq f(v),
\end{equation}
for every $t\in(0,\infty)$ and $v\in\dom f.$ See~\cite[(4.0.13)]{ambrosio}.

\subsection{Weak convergence in Hadamard spaces} 
Here we recall the definition and basic properties of the weak convergence in Hadamard spaces. For a systematic account, the reader is referred to~\cite[Section 3]{semigroups}.

We shall say that a bounded sequence $(x_n)\subset\hs$ converges \emph{weakly} to a point $x\in\hs$ if $P_\gamma x_n\to x$ as $n\to\infty$ for every geodesic $\gamma\geo$ with $\gamma(0)=x.$ We use the notation $x_n\wto x.$

If there is a subsequence $(x_{n_p})$ of $(x_n)$ such that $x_{n_p}\wto z$ for some $z\in \hs,$ we say that $z$ is a \emph{weak cluster point} of the sequence $(x_n).$ The following important result first appeared in~\cite[Theorem 2.1]{jost94}.
\begin{prop} \label{prop:weakcluster}
Each bounded sequence has a weakly convergent subsequence, or in other words, each bounded sequence has a weak cluster point.
\end{prop}
\begin{proof}
See~\cite[Theorem 2.1]{jost94} or~\cite[Proposition 3.2]{semigroups}.
\end{proof}
Like in Hilbert spaces, convex closed sets are (sequentially) weakly closed.
\begin{lem}
Let $C\subset \hs$ be a closed convex set. If $(x_n)\subset C$ and $x_n\wto x\in \hs,$ then $x\in C.$ \label{lem:wclosure}
\end{lem}
\begin{proof}
See~\cite[Lemma 3.7]{semigroups}.
\end{proof}

\begin{defi}
We shall say that a function $f\fun$ is \emph{weakly lsc} at a~given point $x\in\dom f$ if
\begin{equation*}\liminf_{n\to\infty} f(x_n)\geq f(x),\end{equation*}
for each sequence $x_n \wto x.$ We say that $f$ is weakly lsc if it is lsc at each $x\in\dom f.$
\end{defi}

\begin{lem} \label{lem:convexlsc}
If $f\fun$ is a convex lsc function, then it is weakly lsc.
\end{lem}
\begin{proof}
See~\cite[Lemma 3.9]{semigroups}.
\end{proof}

\subsection{Nonexpansive mappings in Hadamard spaces} 
The proof of the Lie-Trotter-Kato formula relies upon the notion of a resolvent associated to a family of nonexpansive maps; see \cite[Lemma~3.1]{stoj}.
\begin{defi}\label{def:resmap}
Let $(\hs,d)$ be an Hadamard space and $\l(F_\rho\r)_{\rho>0}$ be a family of nonexpansive maps $F_\rho\map.$ Given $\lambda,\rho\in(0,\infty)$ and $x\in\hs,$ the map
\begin{equation}
y\mapsto \frac1{1+\frac\lambda\rho}x+\frac{\frac\lambda\rho}{1+\frac\lambda\rho}F_\rho y,\qquad y\in\hs,
\end{equation}
is a contraction with Lipschitz constant $\frac{\frac\lambda\rho}{1+\frac\lambda\rho}$ and hence has a unique fixed point, which is denoted $R_{\lambda,\rho}x.$ The mapping $x\mapsto R_{\lambda,\rho}x$ is called the \emph{resolvent} of the family $\l(F_\rho\r)_{\rho}.$
\end{defi}
The resolvent $R_{\lambda,\rho}$ is obviously a nonexpansive mapping. The following important approximation theorem will be invoked in the proof of Theorem~\ref{thm:stoj}.
\begin{thm} \label{thm:approx}
Let $(\hs,d)$ be an Hadamard space and $f\fun$ be of the form~\eqref{eq:f}. We use Notation~\ref{not:s}.
\begin{enumerate}
 \item Let $R_{\lambda,\rho}$ be the resolvent of $F_\rho\as\l(J_\rho^k\circ\dots\circ J_\rho^1 \r).$ If we have
\begin{subequations}
\begin{align*}
J_\lambda (x) & =  \lim_{\rho\to0+} R_{\lambda,\rho}(x)  ,\\
\intertext{for every $x\in\cldom f$ and $\lambda\in(0,\infty),$ then,}
S_t(x) & =  \lim_{n\to\infty} \l(J_{\frac{t}{n}}^k\circ\dots\circ J_{\frac{t}{n}}^1\r)^{(n)}(x), 
\end{align*}
\end{subequations}
for every $x\in\cldom f$ and $t\in[0,\infty),$ and the convergence is uniform with respect to $t$ on each compact subinterval of $[0,\infty).$
\item Let now $R_{\lambda,\rho}$ be the resolvent of $F_\rho\as\l(S_\rho^k\circ P_k \circ\dots\circ S_\rho^1\circ P_1\r).$ If we have
\begin{subequations}
\begin{align*}
J_\lambda (x) & =  \lim_{\rho\to0+} R_{\lambda,\rho}(x)  ,\\
\intertext{for every $x\in\cldom f$ and $\lambda\in(0,\infty),$ then,}
S_t(x) & =  \lim_{n\to\infty} \l(S_{\frac{t}{n}}^k\circ P_k \circ\dots\circ S_{\frac{t}{n}}^1\circ P_1\r)^{(n)}(x), 
\end{align*}
\end{subequations}
for every $x\in\cldom f$ and $t\in[0,\infty).$
\end{enumerate}
\end{thm}
\begin{proof}
See~\cite[Theorem 3.13]{stoj}.
\end{proof}



\section{Proof of the main result} \label{sec:proof}

In this section, we give the promised alternative proof of Theorem~\ref{thm:stoj}. For the readers' convenience, we recall the statement here.
\begin{thm}[Theorem~\ref{thm:stoj} above]
Let $(\hs,d)$ be an Hadamard space and $f\fun$ be of the form~\eqref{eq:f}. We use Notation~\ref{not:s}. Then, for every $t\in[0,\infty)$ and $x\in\cldom f,$ we have
\begin{subequations}
\begin{align}
S_t(x) & =  \lim_{n\to\infty} \l(J_{\frac{t}{n}}^k\circ\dots\circ J_{\frac{t}{n}}^1\r)^{(n)}(x) , \label{i:stoj:i} \\
\intertext{and, if moreover there exists $1\leq l\leq k$ such that $\dom f_l$ is locally compact, then}
S_t(x) & =  \lim_{n\to\infty} \l(S_{\frac{t}{n}}^k\circ P_k \circ\dots\circ S_{\frac{t}{n}}^1\circ P_1\r)^{(n)}(x) . \label{i:stoj:ii}
\end{align}
\end{subequations}
\end{thm}
\begin{proof}
We mix various facts derived in~\cite{stoj} and employ the weak convergence as appropriate.

We first show~\eqref{i:stoj:i}. By Theorem~\ref{thm:approx}, it suffices to show $R_{\lambda,\rho}x\to J_\lambda x$ as $\rho\to 0,$ where $R_{\lambda,\rho}$ now corresponds to the choice $F_\rho\as\l(J_\rho^k\circ\dots\circ J_\rho^1 \r).$ Put $x_0(\rho)\as R_{\lambda,\rho} x$ and $x_j(\rho)\as J_\rho^j x_{j-1}(\rho)$ for $j=1,\dots,k.$ By the definition of $R_{\lambda,\rho}$ we have 
\begin{equation} \label{eq:convcomb}
x_0(\rho)=\frac{1}{1+\frac{\lambda}{\rho}} x+\frac{\frac{\lambda}{\rho}}{1+\frac{\lambda}{\rho}}x_k(\rho),
\end{equation}
and consequently,
\begin{align} 
d\l(x,x_k(\rho)\r) & =\frac{\rho+\lambda}{\lambda}d\l(x,x_0(\rho)\r), \label{eq:1} \\
\intertext{together with,}
d\l(x_0(\rho),x_k(\rho)\r) & =\frac{\rho}{\lambda}d\l(x,x_0(\rho)\r). \label{eq:2} 
\end{align}

Applying~\eqref{eq:resestim} for each $f_j,$ with $j=1,\dots,k,$ and summing the resulting inequalities up, we arrive~at
\begin{equation} \label{eq:estim2}
2\rho f(v) \geq 2\rho\sum_{j=1}^k f_j\l(x_j(\rho)\r)+d\l(x_k(\rho),v\r)^2 -d\l(x_0(\rho),v\r)^2 +\sum_{j=1}^k d\l(x_{j-1}(\rho),x_j(\rho)\r)^2,
\end{equation}
for every $v\in\dom f.$ The inequality~\eqref{eq:cat} yields
\begin{equation}
\frac{\frac{\lambda}{\rho}}{1+\frac{\lambda}{\rho}}d\l(v,x_k(\rho)\r)^2  \geq d\l(v,x_0(\rho)\r)^2-\frac1{1+\frac{\lambda}{\rho}}d(v,x)^2  +\frac{\frac{\lambda}{\rho}}{\l(1+\frac{\lambda}{\rho}\r)^2}d\l(x,x_k(\rho)\r)^2. \label{eq:estim2.5}
\end{equation}
Combining this inequality with~\eqref{eq:1} and~\eqref{eq:estim2} gives after some elementary calculations that
\begin{equation} \label{eq:estim3}
2\lambda f(v) \geq 2\lambda\sum_{j=1}^k f_j\l(x_j(\rho)\r)+d\l(x_0(\rho),x\r)^2 +d\l(x_0(\rho),v\r)^2 -d(x,v)^2.
\end{equation}
for every $v\in\dom f.$ 

Fix now a sequence $\rho_n\to0.$ We will show that, for every $j=0,1,\dots,k,$ the sequence $\l(x_j\l(\rho_n\r)\r)_n$ is bounded. To this end, apply Lemma~\ref{lem:growth} to obtain $\alpha,\beta\in\rls$ such that
\begin{equation*}
 f_j\l(x_j\l(\rho_n\r)\r)\geq\alpha+\beta d\l(x,x_j\l(\rho_n\r)\r),
\end{equation*}
for every $j=1,\dots,k$ and $n\in\nat.$ Since
\begin{equation*}
 d\l(x_0\l(\rho_n\r),v\r)^2  \geq \frac12 d\l(x_0\l(\rho_n\r),x\r)^2 - d(x,v)^2,
\end{equation*}
inequality~\eqref{eq:estim3} yields
\begin{equation} \label{eq:estim3a}
\lambda f(v) \geq \lambda\biggl[k\alpha+\beta\sum_{j=1}^k d\l(x,x_j\l(\rho_n\r)\r)\biggr]+\frac34 d\l(x_0\l(\rho_n\r),x\r)^2 - d(x,v)^2
\end{equation}
for every $v\in\dom f.$ The case $\beta\geq0$ is easy, we therefore assume $\beta<0.$ Next, observe that
\begin{align}
d\l(x,x_j\l(\rho_n\r)\r) & \leq d\l(x,J_{\rho_n}^jx\r)+ d\l(J_{\rho_n}^jx,J_{\rho_n}^jJ_{\rho_n}^{j-1}x\r)+\cdots+d\l(J_{\rho_n}^j\cdots J_{\rho_n}^1 x,x_j\l(\rho_n\r) \r) \nonumber\\
& \leq \sum_{i=1}^j d\l(x,J_{\rho_n}^i x\r) +d\l(x,x_0\l(\rho_n\r)\r), \nonumber \\
& \leq L+d\l(x,x_0\l(\rho_n\r)\r) \label{eq:estim5}.
\end{align}
for some $L>0$ and every $n\in\nat$ and $j=1,\dots,k.$ Plugging this inequality into~\eqref{eq:estim3a} gives that $\l(x_0\l(\rho_n\r)\r)_n$ is bounded and by~\eqref{eq:estim5} we get that the sequence $\l(x_j\l(\rho_n\r)\r)_n$ is bounded also for $j=1,\dots,k.$

Then Lemma~\ref{lem:convexbnd} yields that the sequence $\l(f_j\l(x_j\l(\rho_n\r)\r)\r)_n$ is bounded from below and inequality~\eqref{eq:estim3} implies that it is also bounded from above, for every $j=1,\dots,k.$

Consider~\eqref{eq:estim2} with $\rho_n$ and take the limit $n\to\infty$ to obtain via~\eqref{eq:2} that
\begin{equation} \label{eq:limit}
\lim_{n\to\infty} \sum_{j=1}^k d\l(x_{j-1}\l(\rho_{n}\r),x_j\l(\rho_{n}\r)\r)^2 =0. 
\end{equation}

Let $z\in\hs$ be a weak cluster point of~$x_0\l(\rho_n\r)$ and $x_0\l(\rho_{n_p}\r)$ be a sequence weakly converging to~$z.$ Recall that the existence of a weak cluster point was guaranteed by Proposition~\ref{prop:weakcluster}. By~\eqref{eq:limit}, also $x_j\l(\rho_{n_p}\r)$ weakly converges to~$z,$ for each $j=1,\dots,k.$ Consequently, $z\in\cldom f$ due to Lemma~\ref{lem:wclosure}.

Consider next inequality~\eqref{eq:estim3} with $\rho_{n_p}$ and take the limit $p\to\infty.$ Since the functions $f_j,$ with $j=1,\dots,k,$ are weakly lsc by Lemma~\ref{lem:convexlsc}, we obtain $z=J_\lambda x.$ In particular, $z\in\dom f.$ Since $z$ was an arbitrary weak cluster point of~$x_0\l(\rho_n\r),$ we get $x_0\l(\rho_n\r)\wto J_\lambda x.$ Applying once again~\eqref{eq:estim3} with $v\as J_\lambda x,$ gives $x_0\l(\rho_n\r)\to J_\lambda x$ as $n\to\infty.$ This finishes the proof of~\eqref{i:stoj:i}.

Next we show~\eqref{i:stoj:ii}. By Theorem~\ref{thm:approx}, it again suffices to show $R_{\lambda,\rho}x\to J_\lambda x$ as $\rho\to 0,$ where $R_{\lambda,\rho}$ now corresponds to the choice $F_\rho\as\l(S_\rho^k\circ P_k \circ\dots\circ S_\rho^1\circ P_1\r).$

Put $x_0(\rho)\as R_{\lambda,\rho} x$ and $x_j(\rho)\as\l(S_\rho^j\circ P_j\r) x_{j-1}(\rho)$ for $j=1,\dots,k.$ Since we again have~\eqref{eq:convcomb}, the same arguments as above yield that~\eqref{eq:1}, \eqref{eq:2} and also~\eqref{eq:estim2.5} hold true.

Applying~\eqref{eq:evi} with $f_j,$ for $j=1,\dots,k$ and summing the resulting inequalities up gives
\begin{equation} \label{eq:estim4}
f(v) \geq \frac1{2\rho}d\l(x_k(\rho),v\r)^2-\frac1{2\rho}d\l(x_0(\rho),v\r)^2+\sum_{j=1}^k f_j\l(x_j(\rho)\r),
\end{equation}
for every $v\in\dom f.$ Combining this inequality with~\eqref{eq:1} and~\eqref{eq:estim2.5} gives after some elementary calculations that~\eqref{eq:estim3} holds as well.

Fix now a sequence $\rho_n\to0.$ For every $j=0,1,\dots,k,$ the sequence $\l(x_j\l(\rho_n\r)\r)_n$ is bounded by the same argument as above. Then Lemma~\ref{lem:convexbnd} yields that the sequence $\l(f_j\l(x_j\l(\rho_n\r)\r)\r)_n$ is bounded from below and inequality~\eqref{eq:estim3} implies that it is also bounded from above for every $j=1,\dots,k.$

The inequality~\eqref{eq:evi} implies that
\begin{equation*}d\l(x_j\l(\rho_n\r),v\r)^2-d\l(P_jx_{j-1}\l(\rho_n\r),v\r)^2 +2\rho_n f_j\l(x_j\l(\rho_n\r)\r)\leq 2\rho_n f_j(v),\end{equation*}
for each $j=1,\dots,k$ and $v\in\dom f_j.$ Take the limit $n\to\infty$ to obtain
\begin{equation} \label{eq:limsup1}
\limsup_{n\to\infty} \l[ d\l(x_j\l(\rho_n\r),v\r)^2-d\l(P_jx_{j-1}\l(\rho_n\r),v\r)^2 \r] \leq 0,
\end{equation}
for each $j=1,\dots,k$ and $v\in\dom f_j.$ 

Let us show that
\begin{equation} \label{eq:limitadhoc4}
 d\l(x_j\l(\rho_n\r),v\r)^2-d\l(x_0\l(\rho_n\r),v\r)^2\to0,\quad\text{as } n\to\infty,
\end{equation}
for each $j=1,\dots,k$ and $v\in\cldom f.$ On one hand, we estimate
\begin{align}
 d\l(x_j\l(\rho_n\r),v\r)-d\l(x_0\l(\rho_n\r),v\r) & \leq d\l(S_{\rho_n}^jP_j v,v\r)+d\l(S_{\rho_n}^jP_j S_{\rho_n}^{j-1}P_{j-1} v,S_{\rho_n}^jP_j v\r) \nonumber \\
& \quad +\cdots+ d\l(S_{\rho_n}^jP_j\cdots S_{\rho_n}^1P_1 v,x_j\l(\rho_n\r)\r) \nonumber \\ & \quad  - d\l(x_0\l(\rho_n\r),v\r) \nonumber \\ & \leq \sum_{i=1}^j d\l(S_{\rho_n}^iP_i v,v\r) +d\l(x_0\l(\rho_n\r),v\r)-d\l(x_0\l(\rho_n\r),v\r) \nonumber \\ & = \sum_{i=1}^j d\l(S_{\rho_n}^iP_i v,v\r), \label{eq:adhocalign1}
\end{align}
and on the other hand,
\begin{align}
 d\l(x_j\l(\rho_n\r),v\r)-d\l(x_0\l(\rho_n\r),v\r) & \geq d\l(x_k\l(\rho_n\r),S_{\rho_n}^kP_k\cdots S_{\rho_n}^{j+1}P_{j+1}v\r)-d\l(x_0\l(\rho_n\r),v\r) \nonumber \\ & \geq d\l(x_k\l(\rho_n\r),v\r)- d\l(v,S_{\rho_n}^kP_k\cdots S_{\rho_n}^{j+1}P_{j+1}v\r)\nonumber \\ &\quad -d\l(x_0\l(\rho_n\r),v\r) \nonumber \\ &\geq - d\l(x_k\l(\rho_n\r),x_0\l(\rho_n\r)\r)- d\l(v,S_{\rho_n}^kP_k v\r)-\cdots \nonumber \\ &\quad -d\l(S_{\rho_n}^kP_k\cdots S_{\rho_n}^{j+2}P_{j+2}v,S_{\rho_n}^kP_k\cdots S_{\rho_n}^{j+1}P_{j+1} v\r) \nonumber \\ & \geq -d\l(x_k\l(\rho_n\r),x_0\l(\rho_n\r)\r) - \sum_{i=j+1}^k d\l(v,S_{\rho_n}^iP_i v\r). \label{eq:adhocalign2}
\end{align}
Observe that $d\l(x_k\l(\rho_n\r),x_0\l(\rho_n\r)\r)=\frac{\rho_n}{\lambda}d\l(x,x_0\l(\rho_n\r)\r)\to0$ as $n\to\infty$ from~\eqref{eq:2} and furthermore
\begin{align*}
 d\l(x_j\l(\rho_n\r),v\r)^2-d\l(x_0\l(\rho_n\r),v\r)^2 & = \l[d\l(x_j\l(\rho_n\r),v\r)-d\l(x_0\l(\rho_n\r),v\r)\r] \\ &\quad \cdot\l[d\l(x_j\l(\rho_n\r),v\r)+d\l(x_0\l(\rho_n\r),v\r)\r],
\end{align*}
which along with~\eqref{eq:adhocalign1} and~\eqref{eq:adhocalign2} gives~\eqref{eq:limitadhoc4}.

Combining~\eqref{eq:limsup1} and~\eqref{eq:limitadhoc4} gives, we obtain
\begin{equation*}\limsup_{n\to\infty} \l[ d\l(x_{j-1}\l(\rho_n\r),v\r)^2-d\l(P_jx_{j-1}\l(\rho_n\r),v\r)^2 \r] \leq 0,\end{equation*}
for each $j=1,\dots,k$ and $v\in\dom f.$ The inequality~\eqref{eq:pyth} now reads
\begin{equation*} d\l(P_j x_{j-1}\l(\rho_n\r),v\r)^2 + d\l(x_{j-1}\l(\rho_n\r),P_j x_{j-1}\l(\rho_n\r)\r)^2\leq d\l(x_{j-1}\l(\rho_n\r),v\r)^2,\end{equation*}
for every $v\in\cldom f_j.$ Therefore
\begin{equation} \label{eq:limsup2}
\limsup_{n\to\infty} d\l(x_{j-1}\l(\rho_n\r),P_j x_{j-1}\l(\rho_n\r)\r)^2 \leq 0,
\end{equation}
for every $j=1,\dots,k.$

Recall that $\l(x_l\l(\rho_n\r)\r)_n\subset\dom f_l$ is a bounded sequence and $\dom f_l$ is locally compact. Let $z\in\dom f_l$ be a cluster point~$x_l\l(\rho_n\r)$ and $x_l\l(\rho_{n_p}\r)$ be a sequence converging to~$z.$ If $l=k,$ then by~\eqref{eq:2} we get $x_0\l(\rho_{n_p}\r)\to z$ and furthermore, by~\eqref{eq:limsup2} we also know $P_1 x_0\l(\rho_{n_p}\r)\to z.$ Therefore $z\in\cldom f_1$ and inequality~\eqref{eq:limsup1} together with an easy approximation argument yield $x_1\l(\rho_{n_p}\r)\to z.$ Repeating this procedure we obtain that $z\in\cldom f$ and $x_j\l(\rho_{n_p}\r)$ converges to~$z,$ for $j=1,\dots,k.$ If $l<k,$ we use the same argument and again obtain that $z\in\cldom f$ and $x_j\l(\rho_{n_p}\r)$ converges to~$z,$ for $j=1,\dots,k.$

Consider next inequality~\eqref{eq:estim3} with $v\as J_\lambda x$ and $\rho\as \rho_{n_p}$ and take the limit $p\to\infty.$ Since the functions $f_j,$ with $j=1,\dots,k,$ are lsc, we obtain $z=J_\lambda x.$ In particular, $z\in\dom f.$ Since $z$ was an arbitrary cluster point of~$x_0\l(\rho_n\r),$ we get $x_0\l(\rho_n\r)\to J_\lambda x.$ Applying once again~\eqref{eq:estim3} with $v\as J_\lambda x$ and $\rho\as \rho_n$ gives $x_0\l(\rho_n\r)\
\to J_\lambda x$ as $n\to\infty.$ This finishes the proof of~\eqref{i:stoj:ii}. 
\end{proof}



\bibliographystyle{siam}
\bibliography{lie-trotter-kato}

\end{document}